\documentclass{article}

\usepackage{PRIMEarxiv}

\usepackage[utf8]{inputenc} 
\usepackage[T1]{fontenc}    
\usepackage{hyperref}       
\usepackage{url}            
\usepackage{booktabs}       
\usepackage{nicefrac}       
\usepackage{microtype}      
\usepackage{lipsum}
\usepackage{fancyhdr}       

\usepackage{amsthm}
\usepackage{amsfonts}
\usepackage{amsmath}
\usepackage{amssymb}
\usepackage{circledsteps}
\usepackage{graphicx}
\usepackage{hyperref}

\newcommand{\dg }{^{\dagger}}
\newcommand{\R}{\mathbb{R}}
\newcommand{\tr}{\mathrm{tr}}

\renewcommand{\dag}{\dagger}
\newtheorem{prop}{Proposition} 
\newtheorem{theorem}{Theorem} 
\newtheorem{lemma}{Lemma}

\DeclareMathOperator*{\argmin}{argmin} 

\hypersetup{%
    pdftitle={Direct Methods for Calculating Pseudoinverses},
    pdfauthor={Jeff Knisley},
    pdfkeywords={Pseudoinverses, Regression, Sparsity}
    }

\title{Direct Methods for Calculating Pseudoinverses
\thanks{\textit{\underline{Citation}}: 
\textbf{Authors. Title. Pages.... DOI:000000/11111.}} 
}

\author{
  Jeff Knisley \\
  Department of Mathematics and Statistics \\
  East Tennessee State University \\
  Johnson City, TN 3614-0663\\
  \texttt{\{Jeff Knisley\}knisleyj@etsu.edu} \\
  }

\begin{document}
\maketitle

\begin{abstract}
The Moore-Penrose pseudoinverse of a matrix can be defined and calculated using its singular value decomposition.  There are also direct methods for computing matrix pseudo-inverses (those that avoid eigenvalue computations), but these are often rank-revealing, poorly conditioned, or otherwise limited in practice.  In this paper, we demonstrate that direct methods can overcome these limitations.  In particular, we reinterpret several existing direct methods and introduce new variations that are appropriate for large scale or sparse multilinear regression 
applications.
\end{abstract}

\begin{keywords}
{pseudoinverses, regression, sparsity}
\end{keywords}

\section{Introduction}

Many data science and machine learning applications rely on multi-linear
regression and least squares, and often such problems are solved
using a Moore-Penrose pseudoinverse (MP pseudoinverse) \cite{penrose1955generalized}. In principle, all multi-linear regression problem 
can be solved using an MP pseudoinverse, but in practice, iterative least squares solvers such as Paige and Saunders' 
LSQR algorithm are often preferred due to their numerical stability
\cite{paige1982lsqr, grisetti2020least}.  Moreover, although there are direct 
methods for computing the pseudoinverse of a matrix, such as through 
the calculation of explicit inverses in solving the normal equations
\cite{benisrael2003generalized}, preferred methods for 
the calculation of a pseudoinverse tend to rely on the full singular 
value decomposition of a matrix \cite{golub1970singular}.  

In this paper, we make the case for pseudoinverses as solvers for least squares
problems. In addition, we also make the case for direct methods for 
calculating pseudoinverses, including via limit forms.  We do so by addressing key limitations in the use of direct methods, including 
\begin{itemize}
    \item \textbf{Conditioning Issues:}  Even if $A\in\R^{n\times n}$ is not ill-conditioned, the products $A^TA$ and $AA^T$ might be because their condition number is the \emph{square} of the condition number of $A$
    \item \textbf{Rank Revealing Algorithms: } Singular values of a matrix sufficiently close to 0 make rank estimation of the matrix impractical in finite arithmetic where rounding to 0 is unavoidable, so algorithms which require or reveal the exact rank  of a matrix are likewise impractical. 
    \item \textbf{Dense Pseudoinverses of Sparse Matrices: } A difficult issue for all approaches is that pseudoinverses of large, sparse matrices tend to be dense, thus necessitating algorithms that do not need an explicit representation of the pseudoinverse itself. 
\end{itemize}
There are also other issues, such as speed, stability, and similar, but these three are the focus of this article.

In particular, despite the excellent existing methods for least squares
solvers for sparse linear equations, the need for methods to address
sparse regression problems is only increasing \cite{mishra2021accelerating}. 
Large language models, for instance, have increased the need for sparse 
matrix solvers \cite{zaheer2020bigbird,hoefler2021sparsity}. This includes
the need for fault-tolerant direct sparse solvers such as the bidiagonal 
block-tearing methods proposed later in this paper \cite{wang2025schwarz}.

The paper is organized as follows.  Section 2 is a brief review of matrix 
pseudoinverses. Section 3 derives several properties of bidiagonal matrices and section 4 investigates generalizations of Lanczos Bidiagonalization methods.  Section 5 discusses error bounds and implementations.  Finally, section 6 examines an entirely 
different method -- 
a variant of familiar regularization methods -- for both MP and group 
pseudoinverses. Section 7 contains additional insights into why direct methods might be of value and how they might be utilized in large scale least squares and other linear algebra applications.

\section{Brief Review of Pseudoinverses}
Most of this section is covered in full detail 
in \cite{benisrael2003generalized}.  To begin with, if $\R^{m\times n}$ denotes the vector space of real $m\times n$ matrices, then the Moore-Penrose 
pseudo-inverse of $A\in \mathbb{R}^{m\times n}$ is the unique matrix $A^\dagger$ which
satisfies the four properties in Table \ref{Table1}.  
\begin{table}[htb]
    \centering
\begin{tabular}{lllll}
1. & $AA^{\dagger }A=A$ &  & 3. & $AA\dg $ is self-adjoint \\ 
2. & $A\dg AA\dg =A\dg $ &  & 4. & $A\dg A$ is self-adjoint
\end{tabular} 
    \caption{The Four Defining Properties of the Moore Penrose Pseudoinverse}
    \label{Table1}
\end{table}
It is straightforward to show that $\left( A^{T}\right)\dg =\left(
A\dg \right)^{T}$ and that if $A$ is invertible (and thus square) that $
A\dg =A^{-1}.$  

If $b\in \mathbb{R}^{m}$ and $x_{p}=A\dg b,$ then condition 1 implies
that 
$$
A^{T}Ax_{p}=\left( AA\dg A\right)^{T}A\left( A\dg b\right)
=A^{T}\left( AA\dg \right)^{T}AA\dg b
$$
Condition 3 then leads to 
$$
A^{T}Ax_{p}=A^{T}AA\dg AA\dg b=A^{T}AA\dg b=\left( \left(
AA\dg \right)^{T}A\right)^{T}b
$$
Conditions 1 and 3 again imply $A^{T}Ax_{p}=\left( AA\dg A\right)^{T}b=A^{T}b$.  
That is, conditions 1 and 3 are sufficient for $x_{p}=A\dg b$ to be a
solution to the normal equations 
$$
A^{T}Ax=A^{T}b
$$

As a result, conditions 1 and 3 are sufficient for
\begin{equation}
A\dg b\in \argmin_x   \left\Vert Ax-b\right\Vert_{2}^{2} \label{LeastSquares}
\end{equation}
Moreover, we can extend (\ref{LeastSquares}) to show that
\begin{equation}
A\dg B=\argmin_{X}\left\Vert AX-B\right\Vert_{Fro}^{2}
\label{MatrixLS}
\end{equation}
There are many applications where an MP pseudoinverse is applied more than once, often in association with (\ref{MatrixLS}).  For example, if there are multiple matrices $B$ for which (\ref{MatrixLS}) must be solved, a direct method that calculates $A$ only once would be highly advantageous.

However, calculation of an MP pseudoinverse tends to be computationally
expensive. Indeed (\ref{MatrixLS}) is often avoided in applications, and in 
general \cite{golub1970singular}, the typical approach to obtaining 
$A\dg $ from $A$ is via the Singular Value Decomposition (SVD) 
$$
A=U\left[ 
\begin{array}{cc}
\Sigma_{r} & 0 \\ 
0 & 0
\end{array}
\right] V^{T},  \quad \Sigma_{r}=diag\left( \sigma_{1},\ldots ,\sigma_{r}\right)
$$
where $U\in \mathbb{R}^{m\times m}, V\in \mathbb{R}^{n\times n}$ are
orthogonal and $\sigma_{1}\geq \ldots \geq \sigma_{r}>0$ are the singular
values of $A$ (thus implying that $A$ has rank $r$).  The Moore-Penrose
pseudoinverse is correspondingly defined 
$$
A\dg =V\left[ 
\begin{array}{cc}
\Sigma_{r}^{-1} & 0 \\ 
0 & 0
\end{array}
\right]^{T}U^{T}
$$

If we define the \emph{effective condition number} of $A\in \R^{m\times n}$ to be 
$C_{\mathit{eff}}(A) = \sigma_1/\sigma_r$, then 
$$C_{\mathit{eff}}\left(A^TA\right) = \left( \; C_{\mathit{eff}}(A) \; \right)^2$$
Thus,  methods based on either $A^{T}A$ or $AA^{T}$ (or both) can be ill-conditioned even if $A$ itself is not. The SVD approach is often preferred over direct methods because the Golub-Kahan algorithm does not rely on such products \cite{golub1965calculating}. 

There are also direct methods for calculating $A\dg ,$ often based on the
fact that if $C$ is full column rank and $R$ is full row rank that 
$\left( CR\right)\dg =R\dg C\dg $.
Thus, if $C$ is a set of independent columns spanning $ran\left( A\right) $
and $R$ is the non-zero rows of the corresponding reduced row echelon form
(Strangs CR decomposition), then $A=CR$ implies that $A\dg = R\dg C\dg$.  Since $R$ and $C$ are full column and row ranks, respectively, we have
\begin{equation}
A\dg = R^T \left( R R^T\right)^{-1} \left( C^T C\right)^{-1} C^T = R^{T}\left( C^{T}CRR^{T}\right)^{-1}C^{T} \label{MacDuffee}
\end{equation}
which is known as the MacDuffee formula. However, while MacDuffee's formula is sufficient for some applications -- it is the method used in the computer algebra system \textbf{sympy} -- we again have a method that depends on products of matrices and their transposes.

The SVD itself can be computationally expensive, often leading to the
avoidance of the pseudo-inverse concept altogether. Sometimes this avoidance
is via Tikhonov regularization (ridge regression) \cite{fuhry2011new} since 
\begin{equation}
A^{\dagger }=\lim_{\varepsilon \rightarrow 0^{+}}\left( A^{T}A+\varepsilon
I_{n}\right)^{-1}A^{T}=\lim_{\varepsilon \rightarrow 0^{+}}A^{T}\left(
AA^{T}+\varepsilon I_{m}\right)^{-1}  \label{Regularization}
\end{equation}
which can be demonstrated using the uniqueness of $A^{\dagger }$ \cite{benisrael2003generalized}.  

Finally, if $A\in \R^{n\times n}$ is a square matrix, then there exists a smallest positive integer $k$ for which 
$$\mathrm{rank}\left( A^k \right) = \mathrm{rank}\left(A^{k+1}\right)$$
We say that $k$ is the index of $A$, and we define 
$$A^D = \lim_{\varepsilon \rightarrow 0} A^k \left(A^{k+1} + \varepsilon^2 I_n\right)^{-1} $$
which is known as the \emph{Drazin} pseudoinverse of $A$.  The three defining
properties of $A^{D}$ are 
\begin{equation}
A^{k}A^{D}A=A^{k},  \quad A^{D}AA^{D}=A^{D}, \quad  AA^{D}=A^{D}A \label{DrazinProperties}
\end{equation}
which can be used to show that $A^{D}$ is unique to $A$.

For comparison, we define a \emph{super-diagonal } matrix to be any
$m\times n$ matrix whose only non-zero coefficients are on its superdiagonal, 
and for $A\in\R^{m\times n}$ we use the notation $A^{-T}$ to denote the matrix of reciprocals of non-zero coefficients of $A^T$.  If $N\in\R^{n\times n}$ is a superdiagonal matrix, then $N$ is nilpotent and its Moore-Penrose and Drazin inverses are, respectively, 
$$
N\dg = N^{-T}   \quad  \text{ and }  \quad  N^{D}=0
$$
Indeed, the Drazin inverse is often important in applications because it
annihilates the nilpotent part of the Jordan form of a matrix. 

If $k=1$, then the Drazin pseudo-inverse is instead called the 
\emph{group} pseudoinverse and is denoted $A^{\#}$.  Moreover, the three properties (\ref{DrazinProperties}) for the group pseudoinverse are 
$$
AA^{\#}A=A,  \quad  A^{\#}AA^{\#}=A^{\#},  \quad  AA^{\#}=A^{\#}A
$$
If $A^{\#}$ exists (that is, $A$ has an index of $k=1$), then it is unique and has a
(\ref{Regularization}) form of 
$$
A^{\#}=\lim_{\varepsilon \rightarrow 0}A \left( A^{2}+\varepsilon^2
I_{m}\right)^{-1}
$$
If $A$ is
diagonalizable, then the Jordan form has no nilpotent part and the group
inverse is defined by the reciprocals of the non-zero eigenvalues (in
analogy with the Moore-Penrose pseudoinverse).  In particular, if $A$ is self-adjoint, then $A$ has index 1 and $A^{2}=A^{T}A$ is positive semi-definite.  That is, if $A\in\R^{n\times n}$ is a symmetric matrix, then $A^{\#} = A\dg$.

Moreover, the group pseudoinverse of $A^{k}$ is $\left( A^{\#}\right)^{k}$ for all $k\in 
\mathbb{Z}^{+}.$  Thus, if we denote $A^{0}=A^{\#}A,$ then 
$$
A^{s+t}=A^{s}A^{t}  \qquad \forall s,t\in \mathbb{Z}
$$
and the ``powers'' of $A$ constitute an
abelian group (hence the name).  For this reason, $A^{\#}$ is often called
the \emph{group inverse} (rather then the group pseudoinverse).  

\section{Decompositions of Bidiagonal Matrices}
Both pseudoinverse and least squares methods tend to apply bidiagonalization algorithms to $A\in\R^{m\times n}$ to produce orthogonal transformations $U,V$ for which $A=UBV^{T}$
and $B$ is bidiagonal \cite{golub2013matrix}.  For example, SVD methods for pseudoinverses tend to rely on bidiagonalization followed by a limiting process with Given's rotations that ``chases'' the superdiagonal
coefficients to 0, thus producing the SVD of $A$ from which the pseudo-inverse is produced.   

However, the bidiagonal matrix $B$ itself is a desirable target for direct calculation of a pseudo-inverse.  It is upper triangular with a
well-defined nilpotent part; and also we have 
$$
\left( UBV^{T}\right)\dg =VB\dg U^{T}
$$
Similarly, if $A$ is an $m\times m$ square matrix, then there exists a
unitary matrix $V$ such that $A=VRV^{T}$ where $R$ is upper triangular (the
Schur decomposition).  There also exists an elementary matrix $E$ for which 
$$
R=EBE^{-1}
$$
where $B$ is bidiagonal \cite{benisrael2003generalized}. Since the Drazin pseudoinverse is unique, it follows that if $A$ has index $k,$ then 
$$
A^{D}=VE \;\; B^{D} \;\; E^{-1}V^{T}
$$

That is, the study of pseudoinverses can be reduced to the study of
pseudo-inverses of bidiagonal matrices. Direct methods for calculating pseudoinverses are those that are applied to the bidiagonal matrix in a bidiagonalization-produced decomposition.

For a matrix $A \in \R^{m\times n}$, we let $A^i$ denote either of the Moore Penrose pseudoinverse or the Drazin pseudoinverse. Suppose $A\in \mathbb{R}^{m\times n}$ is block diagonal with blocks $A_{j}\in \mathbb{R}^{m_{j}\times n_{j}},$ $j=1,\ldots ,\ell .$  It is then straightforward to
show that 
$$
A=\left[ 
\begin{array}{ccc}
A_{1} &  &  \\ 
& \ddots &  \\ 
&  & A_{\ell }
\end{array}
\right] \quad \implies \quad A^{i}=\left[ 
\begin{array}{ccc}
A_{1}^{i} &  &  \\ 
& \ddots &  \\ 
&  & A_{\ell }^{i}
\end{array}
\right]
$$
where $A_{j}^{i}\in \mathbb{R}^{n_{j}\times m_{j}}$ for all $j=1,\ldots
,\ell .$  To illustrate, notice that 
$$
B=\left[ 
\begin{array}{ccccc}
\alpha_{1} & \beta_{1} & 0 & 0 & 0 \\ 
0 & \alpha_{2} & \beta_{2} & 0 & 0 \\ 
0 & 0 & 0 & \beta_{3} & 0 \\ 
0 & 0 & 0 & \alpha_{4} & \beta_{4}
\end{array}
\right]
$$
is block diagonal and that 
\begin{eqnarray*}
\left[ 
\begin{array}{ccc}
\alpha_{1} & \beta_{1} & 0 \\ 
0 & \alpha_{2} & \beta_{2}
\end{array}
\right]\dg  &=&\frac{1}{\alpha_{1}^{2}\alpha_{2}^{2}+\alpha_{1}^{2}\beta_{2}^{2}+\beta_{1}^{2}\beta_{2}^{2}}\left[ 
\begin{array}{cc}
\alpha_{1}\alpha_{2}^{2}+\alpha_{1}\beta_{2}^{2} & -\alpha_{1}\alpha_{2}\beta_{1} \\ 
\beta_{1}\beta_{2}^{2} & \alpha_{1}^{2}\alpha_{2} \\ 
-\alpha_{2}\beta_{1}\beta_{2} & \alpha_{1}^{2}\beta_{2}+\beta_{1}^{2}\beta_{2}
\end{array}
\right] \allowbreak ,      \\
\left[ 
\begin{array}{cc}
\beta_{3} & 0 \\ 
\alpha_{4} & \beta_{4}
\end{array}
\right]\dg  &=&\frac{1}{\beta_{3}\beta_{4}}\left[ 
\begin{array}{cc}
\beta_{4} & 0 \\ 
-\alpha_{4} & \beta_{3}
\end{array}
\right]
\end{eqnarray*}
The Moore-Penrose pseudo inverse of $A$ itself is subsequently 
$$
B\dg =\left[ 
\begin{array}{cc}
\left[ 
\begin{array}{ccc}
\alpha_{1} & \beta_{1} & 0 \\ 
0 & \alpha_{2} & \beta_{2}
\end{array}
\right]\dg  & 0 \\ 
0 & \left[ 
\begin{array}{cc}
\beta_{3} & 0 \\ 
\alpha_{4} & \beta_{4}
\end{array}
\right]\dg 
\end{array}
\right]
$$

Moreover, any $m\times n$ bidiagonal matrix $B$ with rank $r\le\min(m,n)$ is of the form 
$$
B=\left[ 
\begin{array}{cc}
M & 0 \\ 
0 & 0
\end{array}
\right]
$$
where $M$ is either an $r\times r$ square matrix or an $r\times
\left( r+1\right) $ matrix.  It is straightforward to show that 
$$
B^{i}=\left[ 
\begin{array}{cc}
M^{i} & 0 \\ 
0 & 0
\end{array}
\right] , \quad  i=D,\#,\dagger
$$
Thus, without loss of generality, we assume that $B$ is an $r\times \left(
r+1\right) $ bidiagonal matrix.  The diagonal and superdiagonal
coefficients are denoted $\alpha_{j},\beta_{j}$ respectively, $j=1,\ldots
,m.$ If $B$ is extracted from a bidiagonal matrix with $r=n,$ then we
let $\beta_{r}=0.$ As a result, \emph{bidiagonal deflation via Givens
rotations }can be used to show the following:

\begin{theorem}\label{Theorem1}
    If $B$ is an $m\times \left( m+1\right) $ bidiagonal
matrix, then there exists orthogonal matrices $U_{R},V_{R}$ such that 
$$
B=U_{R}\left[ 
\begin{array}{cc}
C & 0 \\ 
0 & K
\end{array}
\right] V_{R}^{T}
$$
where $C$ is invertible bidiagonal and where $K$ is superdiagonal.  
\end{theorem}

\begin{proof}
For more insight into deflation, we refer the reader to section 8.6 of \cite{golub2013matrix} where the algorithm is explained
in detail.  Here our proof simply demonstrates that $U_{R},V_{R}$ are
products of Given's rotations $G_{j,k}\left( \theta \right) ,$ which are
rotations through angle $\theta $ of rows $j$ and $k$ with all other rows
fixed. The $j^{th}$ row of $B$ is of the form 
$$
\mathbf{R}_{j}=\left[ 0\ldots 0,\alpha_{j},\beta_{j},\ldots 0\right]  
$$
Suppose $\alpha_{j}=0$ but that $\alpha_{\nu }\neq 0$ for $\nu =j+1,\ldots
,\ell .$ Suppose also that $\alpha_{\nu }=0$ for all $\nu >\ell .$  For
some choice of $c,s\in \mathbb{R}$ such that $c^{2}+s^{2}=1,$ a Given's
rotation $G_{j,j+1}\left( \theta_{1}\right) $ produces a new row $j$ of
$$
\mathbf{R}_{j}^{\prime }=c\mathbf{R}_{j}-s\mathbf{R}_{j+1}=\left[ 0\ldots
0,c\alpha_{j},c\beta_{j}-s\alpha_{j+1},s\beta_{j+1,}\ldots 0\right]
$$
and $c,s$ can be chosen such that $\mathbf{R}_{j}^{\prime }$ only has at
most non-zero coefficient $s\beta_{j+1}$ in column $j+2$ (the
``bulge'' ). The new $j+1$ row is 
$$
\mathbf{R}_{j+1}^{\prime }=s\mathbf{R}_{j}+c\mathbf{R}_{j+1}=\left[ 0\ldots
0,s\alpha_{j},c\alpha_{j+1}+s\beta_{j},c\beta_{j+1},\ldots 0\right]
$$
which preserves bidiagonality since $\alpha_{j}=0.$  A Given's rotation $
G_{j,j+2}\left( \theta_{2}\right) $ moves the bulge in row $j$ to index $
j+3,$ and so on.  The bidiagonal structure is thus preserved in rows where $
\alpha_{k}\neq 0.$  We continue ``chasing
right'' until the bulge is pushed to column $m+2$, and then
we permute row $j$ to the bottom.  If $\beta_{j-1}=0,$ then the $j^{th}$
column is a column of zeros which can be permuted to the final $m+2$ column
(which is dropped). If $\beta_{j-1}\neq 0,$ then we ``chase
up'' with Given's rotations until the $j^{th}$ column is
all zeros.  We continue until we obtain a matrix $B^{\prime }$ for which $
\alpha_{j}^{\prime }\neq 0$ for $j<r$ and $\alpha_{j}=0$ for all $j\geq r$
for some non-negative integer $r.$  If $\beta_{r}^{\prime }=0,$ then we
have the desired block structure.  If not, then we chase $\beta_{r}^{\prime }$ up until the $r+1$ column is a column of zeros. 
\end{proof}
   
Thus, if $A\in \mathbb{R}^{m\times n},$ then there exists $U\in \mathbb{R}^{m\times m}$ and $V\in \mathbb{R}^{n\times n}$ such that 
$$
A=U\left[ 
\begin{array}{cc}
C & 0 \\ 
0 & K
\end{array}
\right] V^{T}
$$
where $C$ is invertible upper bidiagonal and where $K$ is superdiagonal.  The
Moore-Penrose pseudoinverse is subsequently 
$$
A\dg =V\left[ 
\begin{array}{cc}
C^{-1} & 0 \\ 
0 & K\dg 
\end{array}
\right] U^{T}
$$
where $K\dg = K^{-T}$ is the subdiagonal matrix of reciprocals of non-zero super
diagonal coefficients of $K.$ Since $C$ is invertible bidiagonal, there is a formula
for calculating $C^{-1}$ or alternatively, the action of $C^{-1}$ can be
implemented via back-substitution \cite{higham2023power}. 

There also ``in-place'' methods for
computing a pseudoinverse of an $r\times \left( r+1\right) $ bidiagonal
matrix $B$ with diagonal and superdiagonal coefficients $\alpha_{j},\beta_{j}$ for $j=1,\ldots ,r.$  Superdiagonal coefficients $\beta_{j}=0$
split $B$ into a block diagonal matrix $B=diag\left( B_{1},\ldots ,B_{\ell
}\right) $. Thus, we only need $B\dg $ for blocks of the form 
$$
B_{k}=\left[ 
\begin{array}{cccc}
a_{k1} & b_{k1} &  &  \\ 
& \ddots & \ddots &  \\ 
&  & a_{k,m_{k}-1} & b_{k,m_{k}-1} \\ 
&  &  & a_{k,m_{k}\phantom{+1}}
\end{array}
\right]
$$
where $b_{kj}\neq 0$ for all $j=1,\ldots ,m_{k}-1.$ Moreover, diagonal
coefficients $a_{kj}=0$ split a block $B_{k}$ into a block diagonal matrix with often non-square diagonal blocks 
$M_{ki},$ $i=1,\dots ,\nu_{k}$.  That is, $B_k$ blocks are of the form 
$$
B_{k}=\left[ 
\begin{array}{cccc}
M_{k1} & 0 & \ldots & 0 \\ 
0 & M_{k2} & \ddots & \vdots \\ 
\vdots & \ddots & \ddots & 0 \\ 
0 & \ldots & 0 & M_{kv_{k}}
\end{array}
\right]
$$
Each block $M_{ki}$ is in one of only 3 possible forms:
\begin{enumerate}
\item Square $\mu \times \mu $ lower bidiagonal with non-zero diagonal
coefficients from the superdiagonal of $B_{k}.$

\item Rectangular with $\mu $ rows and $\mu +1$ columns, which can be
written in the form $\left[ a,R\right] $ where $a$ is the first column and $
R $ is invertible lower bidiagonal with non-zero diagonal coefficients from
the superdiagonal of $B_{k}.$

\item Rectangular with $\mu +1$ rows and $\mu $ columns, which can be
written in the form $\left[ R,a\right]^{T}$ where $a$ is the last row and $
R $ is invertible bidiagonal with non-zero diagonal coefficients from the
superdiagonal of $B_{k}.$
\end{enumerate}

\begin{figure}[p!]
    \centering
    \includegraphics[width=\textwidth, trim=2cm 7cm 2cm 3cm ]{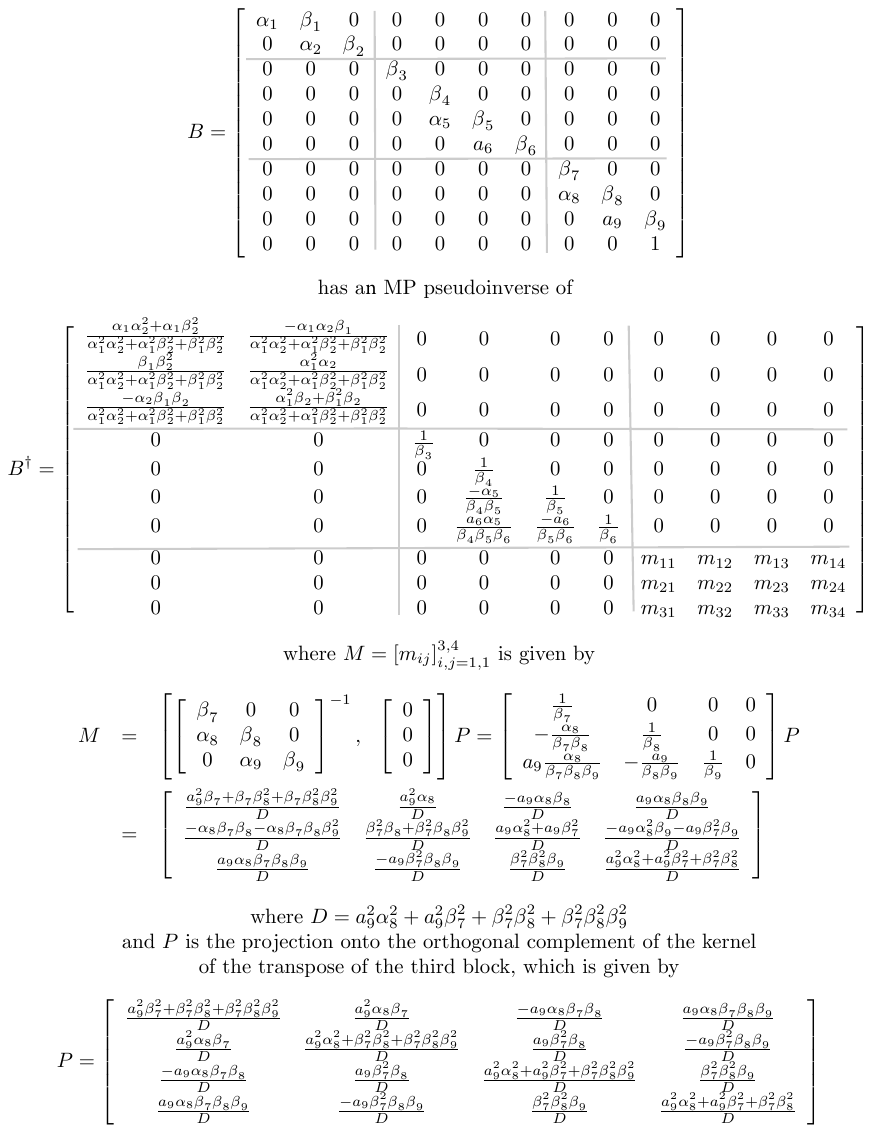}
    \caption{The in-place MP pseudoinverse of the bidiagonal matrix $B$}
    \label{Figure1}
\end{figure}

\noindent To illustrate, the bidiagonal matrix in Figure \ref{Figure1} has been
partitioned into these three forms.  The pseudo-inverses of each type of block can be produced explicitly in no
more than $3\mu $ operations per block. For example, assuming the
coefficients $\alpha_{j},$ $\beta_{j}$ are non-zero in Figure \ref{Figure1}, the
Moore-Penrose pseudo-inverse of $B$ is as shown in Figure \ref{Figure1}.

Specifically, after splitting at
the coefficients $\beta_{j}=0,$, we split again at the coefficients $\alpha_{j}=0$
to obtain a block matrix whose diagonal is made up of $\mu \times \mu $
square invertible bidiagonal matrices or rectangular matrices with 
one more row than column or one more column than row.  That is, a $\mu
\times \mu $ block is of the form 
$$
B_{k}=\left[ 
\begin{array}{cccc}
\beta_{j} &  &  &  \\ 
\alpha_{j+1} & \beta_{j+1} &  &  \\ 
& \ddots & \ddots &  \\ 
&  & \alpha_{\ell } & \beta_{\ell }
\end{array}
\right]
$$
and the inverse of $B_{k}$ in closed form is \cite{higham2023power}
$$
\left( B_{k}^{-1}\right)_{i,j}=\frac{1}{\beta_{j}}\prod_{\nu
=j}^{i-1}\left( \frac{-\alpha_{\nu +1}}{\beta_{\nu }}\right) ,   i\geq j
$$
A rectangular $\mu \times \left( \mu +1\right) $ block is of the form $B_{k}=
\left[ a,R\right] $ where $R$ is an invertible lower bidiagonal matrix. Its
Moore-Penrose pseudo inverse is of the form 
$$
B_{k}\dg =P\left[ 
\begin{array}{c}
R^{-1} \\ 
0
\end{array}
\right]
$$
where $P$ is the orthogonal projection onto the orthogonal complement onto 
$$
\ker \left( B_{k}\right) =\left\{ c\left[ 
\begin{array}{c}
-1 \\ 
R^{-1}a
\end{array}
\right]   |  c\in \mathbb{R}\right\}
$$
The projection $P$ can be calculated explicitly, thus leading to the
following proposition: 

\begin{prop}
    If $B_{k}=\left[ a,R\right] $ is a $\mu \times \left( \mu
+1\right) $ matrix where $a\in \mathbb{R}^{\mu }$ and $R\in \mathbb{R}^{\mu
\times \mu }$ is invertible, then 
$$ B_{k}\dg =\left[ \begin{array}{c} cv^{T} \\ 
R^{-1}-czv^{T}
\end{array}
\right]
$$
where $z=R^{-1}a,$ $c=\left( 1+z^{T}z\right)^{-1},$ and $v^{T}=z^{T}R^{-1}.$
\end{prop}
\begin{proof}
First, we notice that 
$$
B_{k}B_{k}\dg =
\begin{array}{c}
\left[ 
\begin{array}{cc}
a & R
\end{array}
\right] \\ 
\end{array}
\left[ 
\begin{array}{c}
cv^{T} \\ 
R^{-1}-czv^{T}
\end{array}
\right] =\frac{av^{T}}{1+z^{T}z}+I-\frac{Rzv^{T}}{1+z^{T}z}
$$
and that 
$$
B_{k}\dg B_{k}=\left[ 
\begin{array}{c}
cv^{T} \\ 
R^{-1}-czv^{T}
\end{array}
\right] 
\begin{array}{c}
\left[ 
\begin{array}{cc}
a & R
\end{array}
\right] \\ 
\end{array}
=\left[ 
\begin{array}{cc}
cv^{T}a & cv^{T}R \\ 
R^{-1}a-czv^{T}a & I-czv^{T}R
\end{array}
\right]
$$
However, $v^{T}a=z^{T}R^{-1}a=z^{T}z$ and $v^{T}R=z^{T}.$  Thus, we have 
$$
B_{k}\dg B_{k}=\left[ 
\begin{array}{cc}
cz^{T}z & cz^{T} \\ 
z-czz^{T}z & I-czz^{T}
\end{array}
\right]
$$
But $\left( 1+z^{T}z\right) c=1$ implies that $1-z^{T}zc=c.$ Thus, 
$$
B_{k}\dg B_{k}=\left[ 
\begin{array}{cc}
cz^{T}z & cz^{T} \\ 
cz & I-czz^{T}
\end{array}
\right]
$$
which is self-adjoint.  Moreover, 
$$
B_{k}B_{k}\dg B_{k}=
\begin{array}{c}
\left[ 
\begin{array}{cc}
a & R
\end{array}
\right] \\ 
\end{array}
\left[ 
\begin{array}{cc}
cz^{T}z & cz^{T} \\ 
cz & I-czz^{T}
\end{array}
\right] =
\begin{array}{c}
\left[ 
\begin{array}{cc}
cz^{T}za+cRz & acz^{T}+R-cRzz^{T}
\end{array}
\right] \\ 
\end{array}
$$
However, $Rz=a,$ which implies the first coefficient is 
$$
\frac{1}{z^{T}z+1}\left( z^{T}za+a\right) =a
$$
Moreover, $acz^{T}=cRzz^{T},$ thus yielding $B_{k}B_{k}\dg B_{k}=B_{k}.$
 Conditions 2 and 4 are similar.  
\end{proof} 

Alternatively, the Moore-Penrose pseudoinverse $x=B\dg y$ of an $m\times
n$ bidiagonal matrix $B$ applied to a vector $y\in \mathbb{R}^{m+1}$ can be
obtained from either the normal or the dual-normal equations 
$$
B^{T}Bx=B^{T}y\quad or\quad BB^{T}z=b, \;  x=B^{T}z     
$$
If $B$ is $m\times \left( m+1\right) $ bidiagonal, then $BB^{T}$ is an $
m\times m$ tri-diagonal with diagonal coefficients $\alpha_{j}^{2}+\beta_{j}^{2}$ and super-diagonal (and subdiagonal) coefficients $\alpha_{j+1}\beta_{j}.$  Moreover, $BB^{T}$ is block diagonal with blocks
defined where either $\alpha_{j+1}=0$ or $\beta_{j}=0.$  That is, 
$$
BB^{T}=\left[ 
\begin{array}{ccccc}
\alpha_{1}^{2}+\beta_{1}^{2} & \alpha_{2}\beta_{1} &  &  &  \\ 
\alpha_{2}\beta_{1} & \ddots & \ddots &  &  \\ 
& \ddots & \alpha_{j}^{2}+\beta_{j}^{2} & 0 &  \\ 
&  & 0 & \alpha_{j+1}^{2}+\beta_{j+1}^{2} & \ddots \\ 
&  &  & \ddots & \ddots
\end{array}
\right]
$$
Although this approach does not avoid the conditioning issues as discussed
earlier (it results in normal and dual normal equation methods we mentioned earlier), there are still scenarios for which this approach might be highly
advantageous.    

\section{Direct Methods via Lanczos Bidiagonalization}
The decomposition of a bidiagonal matrix $B\in \R^{m\times n}$ into invertible and non-invertible block diagonal components requires a full bidiagonalization $A = U B V^T$.  However, Golub-Kahan Lanczos bidiagonalization of $A\in\R^{m\times n}$ produces a bidiagonal matrix $B_k$ after $k$ iterations such that 
$$A = U_{k+1} B_k V_{k}^T$$
where $U_k$ and $V_k$ have orthonormal columns \cite{simon2000low}. The columns $u_i\in\R^m, v_i\in\R^n$ of $U_k,V_k$, where $i=1,...,k$ are obtained using the Lanczos recurrence relations, 
\begin{eqnarray*} 
    A^Tu_i &=& \alpha_iv_i+\beta_{i-1}v_{i-1} \\
    Av_i &=& \beta_i u_{i+1} + \alpha_i u_i
\end{eqnarray*}
where for a given $y\in \R^n$ we choose $\beta_1\in\R$ such that $\beta_1u_1 = y$ and $u_1$ is a unit vector.  The coefficients $\alpha_i, \beta_i, \; i=1,\ldots,k$ are the diagonal and superdiagonal coefficients, respectively, of the $(k+1)\times k$ \emph{lower} bidiagonal matrix $B_k$.  

Importantly, all the methods in the previous section construct pseudo-inverses that satisfy conditions 2, 3 and 4 in Table \ref{Table1}, which means that convergence of pseudoinverse methods based on the Lanczos approach can be based solely on condition 1 of Table \ref{Table1}.  That motivates the following theorem. 

\begin{theorem} \label{Theorem2}
For $A\in \mathbb{R}^{m\times n},$ suppose for each $k\in \mathbb{Z
}^{+}$ there exists $U_{k}\in \mathbb{R}^{m\times k},$ $V_{k}\in \mathbb{R}^{n\times k}$ such that $U_{k}^{T}U_{k}=V_{k}^{T}V_{k}=I_{k}.$  For $
B_{k}=U_{k}^{T}AV_{k},$ define $R_{k}$ as 
$$A=U_{k}B_{k}V_{k}^{T}+R_{k}$$
and let $X_{k}=V_{k}B_{k}^{\dg }U_{k}^{T}.$  If $R_{k}\rightarrow 0$,
then also $\left\Vert AX_{k}A-A\right\Vert_{F}\rightarrow 0.$   
\end{theorem}

\medskip

\noindent Before we prove our theorem, we first establish the following lemma:

\begin{lemma}
    Let $A_{k}=U_{k}B_{k}V_{k}^{T}.$ Then $X_{k}=A_{k}\dg .$
\end{lemma}

\begin{proof}(lemma): Let $A_{k}=U_{k}B_{k}V_{k}^{T}.$  Then 
$$ A_{k}X_{k}=U_{k}B_{k}V_{k}^{T}\left( V_{k}B_{k}^{\dg }U_{k}^{T}\right)
=U_{k}B_{k}B_{k}\dg U_{k}^{T} $$
which is self-adjoint.  Moreover, 
\begin{eqnarray*}
A_{k}X_{k}A_{k} &=&U_{k}B_{k}B_{k}\dg U_{k}^{T}U_{k}B_{k}V_{k}^{T} \\
&=&U_{k}B_{k}B_{k}\dg B_{k}V_{k}^{T}=U_{k}B_{k}V_{k}^{T}=A_{k}
\end{eqnarray*}
and similarly, $X_{k}A_{k}$ is self adjoint and $X_{k}A_{k}X_{k}=A_{k}$. 
\end{proof}

\medskip

\noindent Now the proof of the theorem: 
\begin{proof}
Let $P_{k}=U_{k}U_{k}^{T}$ and $Q_{k}=V_{k}V_{k}^{T}.$ Then 
$P_{k},Q_{k}$ are orthogonal projections and 
$$
P_{k}AQ_{k}=U_{k}\left( U_{k}^{T}AV_{k}\right)
V_{k}^{T}=U_{k}B_{k}V_{k}^{T}=A_{k}
$$
where $A_{k}=U_{k}B_{k}V_{k}^{T}$ as in the lemma. Thus, $
A=P_{k}AQ_{k}+R_{k},$ which implies that 
\begin{eqnarray*}
\left( P_{k}AQ_{k}\right)^{T}R_{k} &=&Q_{k}A^{T}P_{k}\left(
A-P_{k}AQ_{k}\right) \\
&=&Q_{k}A^{T}P_{k}A-Q_{k}A^{T}P_{k}^{2}AQ_{k}
\end{eqnarray*}
It follow that $\left\langle P_{k}AQ_{k},R_{k}\right\rangle_{F}$ is given
by 
$$
\tr\left( \left( P_{k}AQ_{k}\right)^{T}R_{k}\right) =\tr\left(
Q_{k}A^{T}P_{k}A\right) -\tr\left( Q_{k}A^{T}P_{k}AQ_{k}\right)
$$
But $\tr\left( Q_{k}A^{T}P_{k}AQ_{k}\right) =\tr\left(
Q_{k}^{2}A^{T}P_{k}A\right) $ and $Q_{k}^{2}=Q_{k}.$  Thus, $\left\langle
P_{k}AQ_{k},R_{k}\right\rangle_{F}=0,$ or equivalently, $\left\langle
A_{k},R_{k}\right\rangle_{F}=0$, which implies that
$$
\left\Vert A\right\Vert_{F}^{2}=\left\Vert A_{k}\right\Vert_{F}^{2}+\left\Vert R_{k}\right\Vert_{F}^{2}  \label{sumsquares}
$$

It follows that 
\begin{eqnarray*}
AX_{k}A-A &=&\left( A_{k}+R_{k}\right) A_{k}\dg \left(
A_{k}+R_{k}\right) -A_{k}-R_{k} \\
&=&A_{k}A_{k}\dg A_{k}+R_{k}A_{k}\dg A_{k}+A_{k}A_{k}^{\dag
}R_{k}+R_{k}A_{k}\dg R_{k} -A_{k}-R_{k} \\
&=&R_{k}A_{k}\dg A_{k}+A_{k}A_{k}\dg R_{k}+R_{k}A_{k}^{\dag
}R_{k}-R_{k} \\
&=&R_{k}A_{k}\dg A_{k}+A_{k}A_{k}\dg R_{k}+R_{k}\left( A_{k}^{\dag
}R_{k}-I\right)
\end{eqnarray*}
so that the Frobenius norm yields 
\begin{equation}
\left\Vert AX_{k}A-A\right\Vert_{F}\leq \left\Vert R_{k}A_{k}^{\dag
}A_{k}\right\Vert_{F}+\left\Vert A_{k}A_{k}\dg R_{k}\right\Vert_{F}+\left\Vert R_{k}A_{k}\dg R_{k}-R_{k}\right\Vert_{F}
\label{ProofBound}
\end{equation}
which approaches zero as $R_{k}$ approaches 0.  
\end{proof} 

However, Lanczos bidiagonalization stops if it produces a zero diagonal or subdiagonal coefficient in $B_k$ \cite{golub2013matrix}.  In this case, the Krylov space is invariant and the rank of $A$ has been reached.  Thus, the CK decomposition (Theorem \ref{Theorem1}) is not useful if the exact Lanczos method is utilized.  

Instead, we can ``re-interpret'' Lanczos Bidiagonalization as a direct pseudoinverse method (indeed, a reinterpretation of the LSQR least squares solver \cite{paige1982lsqr}).  In particular, after each iteration $k$, we let $X_k = V_k B_k^{\dagger} U_k^T$ with the goal that $X_k$ converges to $A\dg$ as $k$ increases.  It is straightforward to show that $AX_k$ and $X_kA$ are symmetric.  It is also straightforward to show that $X_k$ is of the form
\begin{equation}
    X_k = X_{k-1} + w_kz_k^T  \label{Xrecur}
\end{equation} 
for vectors $w_k\in\R^n$ and $z_k\in\R^m$.  Indeed, after the $k^{th}$ iteration of the Lanczos method, we use a Given's rotation  to "zero" out the last bidiagonal coefficient \cite{paige1982lsqr}:
$$\begin{bmatrix} c_k & s_k \\ -s_k & c_k \end{bmatrix} \begin{bmatrix} \bar{\rho}_{k-1} & 0 \\ \beta_k & \alpha_k \end{bmatrix} = \begin{bmatrix} \rho_{k-1} & \theta_k \\ 0 & \bar{\rho}_k \end{bmatrix}$$
The vectors $w_k,z_k$ respectively then satisfy 
$$w_k = \frac{1}{\rho_k}\left(v_k - \theta_k w_{k-1}\right), \quad z_k = c_k u_k - s_k z_{k-1}$$
Additionally, we can track convergence using $\|w_k\|_2$, stopping once it is sufficiently close to 0. 

If we use (\ref{Xrecur}) as our starting point, then this reinterpretation is equivalent to finding the stationary point(s) of  
\begin{equation}
    f(w, z) = \| A X_k - I \|_F^2 = \| R_{k} + A w z^T \|_F^2 \label{functionalwz}
\end{equation}
where $\| \ldots \|_F$ is the Frobenius matrix norm (aka, the \emph{trace} norm).  Moreover, if we let $z=Aw$, then (\ref{Xrecur}) becomes 
$$X_k = X_{k-1} + w_k\left(Aw_k\right)^T = X_{k-1} + w_k w_k^TA^T$$
where we require that $w\ne 0, Aw \ne 0$ (else, recursion is an identity).  Notice that 
$$AX_k = AX_{k-1} + (Aw_k)(Aw_k)^T$$
is symmetric.  Moreover, the functional (\ref{functionalwz}) becomes 
$$f(w) = \| R_k + Aw(Aw)^T \|_F^2 = \|R_k + Aww^TA^T\|_F^2$$
Properties of the trace imply that 
\begin{eqnarray*}
    f(w) &=& \|R_k\|_F^2 + 2\tr\left(R\_k Aww^T A^T\right) + \tr\left(Aww^T A^T Aww^T A^T\right) \\
       &=& \|R_k\|_F^2 + 2 w^T A^T R_k Aw + (w^T A^T Aw)^2\\ 
       &=&\|R_k\|_F^2 + 2 w^T A^T R_k Aw + \|Aw\|_2^4
\end{eqnarray*}
from which the gradient is given by 
$$\nabla_w f(w) = 4 A^T R_k Aw + 4 \|Aw\|_2^2 A^T Aw $$
Setting $\nabla_wf = 0$ and substituting $R_k = AX_{k-1} - I$ yields 
\begin{eqnarray}
    A^T (AX_{k-1} - I) Aw &=& - \|Aw\|_2^2 A^T Aw \nonumber \\
    A^T A X_{k-1} Aw &=& (1 - \|Aw\|_2^2) A^T Aw \label{MainOne}
\end{eqnarray}
Equation (\ref{MainOne}) implies that without loss of generality, we can assume that $Aw$ is a unit vector, thus implying $A^TAX_{k-1}Aw = 0$. Since $\ker(A^TA) = \ker(A)$, we must also have $AX_{k-1}Aw =0.$ The product of $A$ and $X_k$ is  
$$AX_k= AX_{k-1} + (Aw)(Aw)^T,$$ 
which is an orthogonal projection. The result is that
$$AX_kAw = AX_{k-1}Aw + Aw(Aw)^TAw = AX_{k-1}Aw + Aw = Aw$$
Thus, any $w\in \R^n$ such that $Aw$ is a unit vector for which $AX_{k-1}(Aw)=0$ is an argmin of $f(w)$.  This leads to the following theorem: 

\begin{theorem}\label{Theorem3}
    For $w_0\in \R^n$, let $X_0 =w_0(Aw_0)^T$ where $Aw_0$ is a unit vector, and then for $k=1,2,3,... $ let 
    $$X_k = X_{k-1} + w_k (Aw_k)^T$$
    where $w_k \in R^n$ such that $Aw_k$ is a unit vector orthogonal to the columns of $AX_{k-1}$. Then $X_k$ converges to a pseudoinverse satisfying conditions 1,2,3 of Table \ref{Table1}.  
\end{theorem}

\begin{proof}
    To prove the first condition in Table \ref{Table1}, we notice that once $k=r$ where $r=\mathrm{rank}(A)$, then $Aw_0,\ldots, Aw_{r-1}$ is an orthonormal basis for the range of $A$.  As a result, 
    $$P = AX_{r-1} = \sum_{i=0}^{r-1} \left(Aw_i\right)\left(Aw_i\right)^T$$
    is the orthogonal projection of $\R^m$ onto the range of $A$, which means that $PA=A$.  Thus, $AX_{r-1}A=PA=A$. 

    The second condition was established above.  For the third condition, we first notice that 
    \begin{eqnarray*}
        X_kAX_k &=& X_{k-1}AX_{k-1} + X_{k-1}Aw_k(Aw_k)^T \\
                & & + w_k(Aw_k)^TAX_{k-1} + w_k(Aw_k)^T(Aw_k)(Aw_k)^T
    \end{eqnarray*}
 By our recursive assumption, $X_{k-1}AX_{k-1} = X_{k-1}$. For the second term $X_{k-1}Aw_k(Aw_k)^T$, we note that $AX_{k-1}(Aw_k) = 0$.  Thus, 
 $$X_{k-1}AX_{k-1}(Aw_k) = 0 \quad \implies \quad X_{k-1}Aw_k = 0$$
 since $X_{k-1}AX_{k-1} = X_{k-1}$. Thus, $X_{k-1}Aw_k(Aw_k)^T=0(Aw_k)^T = 0$. The third term is 
 $$w_k(Aw_k)^T(AX_{k-1})= w_k(Aw_k)^T\left(AX_{k-1}\right)^T=w_k\left(AX_{k-1}Aw_k\right)^T = 0$$
 since $AX_{k-1}$ is symmetric. Thus, we now have 
 $$X_kAX_k = X_{k-1}AX_{k-1} + w_k\left(Aw_k\right)^T(Aw_k)(Aw_k)^T = X_{k-1} + \|Aw_k\|^2_2\; w_k(Aw_k)^T$$
 Since $Aw$ is a unit vector, we have $X_kAX_k = X_k$ for all $k=0,1,2,3...$. That is, each iteration $X_k$ satisfies conditions 2,3 in Table \ref{Table1}. 
\end{proof}

A pseudoinverse of $A\in\R^{m\times n}$ satisfying only conditions 1,2,3 in Table \ref{Table1} is known as a \emph{weak pseudoinverse} and is denoted $A^{(1,2,3)}$  \cite{benisrael2003generalized}.  Theorem \ref{Theorem3} does not indicate how the vectors $w_k\in\R^n$ are selected, but there are numerous methods for doing so.  For example, if $A\in\R^{m\times n}$ has rank $r$, then define $X_r$ using following algorithm:
\begin{enumerate}
    \item Generate $\ell$ vectors $v_{1},\ldots,v_{\ell}$ such that $\ell > r$ and 
    $$\mathrm{span}\left(Av_1,\ldots, Av_{\ell}\right) = \mathrm{ran}(A)$$
    \item Find coefficients (Gram-Schmidt, for example) $c_{ij}$ for which 
    $$q_{i} = \sum_{i=1}^{\ell}c_{ij} Av_{j}, \quad i=1,\ldots,r$$
    is an orthonormal basis for $\mathrm{ran}(A)$.
    \item Define $w_i, i=1,\ldots, r$ via the same $c_{ij}$ coefficients (e.g., at each Gram-Schmidt iteration) as 
    $$w_i = \sum_{j=1}^n c_{ij} v_j$$
    \item Let $X_r = w_1q_1^T + \ldots + w_r q_r^T$ 
\end{enumerate}
Theorem \ref{Theorem3} implies that $X_r$ is a (1,2,3) pseudoinverse of $A$, thus making it suitable for application to (\ref{MatrixLS}) and other least squares problems (and it is a straightforward algorithm to write -- great student coding exercise which is why I do not include the code in the implementation). 

If necessary, condition 4 in Table \ref{Table1} can subsequently be obtained using the fact that all weak pseudoinverses are of the form 
$$A^{(1,2,3)} = \left\lbrace X_r + \left(I-X_rA\right)Z\left(AX_r\right) \;\; \middle| \;\; Z\in\R^{n\times m}\right\rbrace$$
where $X_r$ is the weak pseudoinverse of $A\in\R^{m\times n}$ with rank $r$  produced by our greedy algorithm. 

We can also implement condition 4 in each step of an algorithm implementing Theorem \ref{Theorem3} .  For example, we can choose each ``next'' vector $w_{k+1}$ using a 2-term recursion 
\begin{eqnarray*}
    u_k &=& Aw_k - \beta_{k-1}u_{k-1} \\
    w_{k+1} &=& A^Tu_k - \alpha_k w_k
\end{eqnarray*}
where $\alpha_k, \beta_{k-1}$ are from Lanczos bidiagonalization.  That is, this example is indeed a ``reinterpretation'' of the LSQR method for sparse least squares solvers \cite{paige1982lsqr}.

\section{Implementation and Error Bounds}
In the previous sections, we introduced a general approach of
bidiagonalization of a matrix followed by one of 3 possible methods for
calculating the pseudoinverse of a bidiagonal matrix. The LAPACK commands for 
Householder bidiagonalization are gebrd and orgbr.  Although these are
used for the Golub-Kahan computation of the singular value decomposition of a 
matrix, they are not necessarily available in \textbf{scipy.lapack}. However, 
they can be accessed via a Cython wrapper as is demonstrated at 
\newline \url{https://github.com/appmathdoc/DirectPseudoInverseMethods}

We thus assume the bidiagonalization algorithms have been implemented.  As justified above, our focus here is on $r \times (r+1)$ bidiagonal matrices for which there are three approaches:
\begin{enumerate}
\item \textbf{CK decomposition:}  Given's rotations deflation of a bidiagonal matrix into the bidiagonal form 
$$
B=\left[ 
\begin{array}{cc}
C & 0 \\ 
0 & K
\end{array}
\right] \label{CKform}
$$
where $C$ has a non-zero diagonal and $K$ has a zero diagonal.  

\item \textbf{Dual Normal:} Solving the dual normal equations for $B$ using the natural block-diagonal structure of $BB^{T}.$

\item \textbf{In place:} Splitting $B$ into block diagonal form on the coefficients $\beta_{j}=0$ and then splitting each block on coefficients $\alpha_i=0$ into a block with zeros off the
diagonal and with diagonal matrices a mixture of square and rectangular
blocks.
\end{enumerate}

\noindent We also refer to these as the CK method, the dual normal method, and the inplace method, respectively, for the pseudoinverse of a bidiagonal matrix. 

Our experience suggests that the CK method is widely applicable.  Indeed, if $B$ is in the CK form (\ref{CKform}), then its set of singular values is the union of the singular values of $C$ and the singular values of $K$.  The singular values of $K$ are the absolute values of its nonzero coefficients.  Moreover, the nuclear norm of $B$  is 
$$\|B\|_{\ast} = \|C\|_{\ast} + \|K\|_{\ast}$$
The pseudoinverse $K\dg$ is both fast and sparse, so that there is no real advantage to thresholding the superdiagonal of $B$ other than to avoid reciprocals of small numbers.  However, if $\alpha_i,i=1,\ldots,\ell$ and $\beta_i, i=1,\ldots\ell-1$ are the diagonal and superdiagonal coefficients of $C$, then 
$$\sum_{i=1}^{\ell-1}\left|\beta_i\right|\le\|B\|_{\ast}\le\sum_{i=1}^{\ell}\left|\alpha_i\right|+\sum_{i=1}^{\ell-1}\left|\beta_i\right|$$
Thus, while the CK decompsition is not rank-revealing, thresholding diagonal coefficients of $C$ to 0 is equivalent to thresholding smaller singular values of $B$ to 0.  That is, the CK decomposition has the noise thresholding abilities of the SVD based algorithms. 

However, bidiagonalization followed by
deflation is similar to the Golub-Kahan algorithm for the singular value
decomposition.  The complexity of this method is $O\left(
n^{3}\right) $, as is the Golub-Kahan approach itself. While the deflation
step is $O\left( n^{2}\right),$ the required full bidiagonalization is 
$O\left( n^{3}\right) $, which translates into the CK decomposition being only slightly faster than
the use of the SVD itself to calculate Moore-Penrose pseudo-inverses.

Importantly, the pseudo-inverse of a CK decomposition is 
$$B\dg = \begin{bmatrix} C^{-1} & 0 \\ 0 & K\dg \end{bmatrix}$$
If $A$ is sparse, then its bidiagonalization $B$ is sparse because it is bidiagonal.  The pseudoinverse $K\dg$ is also sparse. Although $C^{-1}$ is not sparse, it is the inverse of the sparse bidiagonal matrix $C$ and can be implemented using back-substitution.  That is, CK decomposition does not require the implementation of a dense matrix, and thus, the CK decomposition is applicable to large, sparse matrices.  

In the dual normal method, the tridiagonal matrix $BB^{T}$ and the related dual normal equations can be addressed with a Thomas algorithm \cite{thomas1949elliptic}, which is an $O(n)$ direct solver for tridiagonal systems in which the subdiagonal of each
invertible block on the block diagonal is moved to zero using Gaussian
elimination and then the resulting upper bidiagonal matrix is inverted via back
substitution.  To ameliorate conditioning issues, Woodbury matrix tearing is used to break solving of systems (or inverting of matrices) into smaller subsystem problems which are then recombined \cite{bunch1974partitioning,hager1989updating}. 

The inplace method calculates an explicit pseudo-inverse of each block ``in place'' in the bidiagonal matrix $B.$  Theoretically, the inplace method is equivalent to CK decomposition, yet similar to the dual normal method, the inplace method does not require a full $O\left( n^{3}\right) $
bidiagonalization.  Moreover, Woodbury methods can be used to keep the memory requirements of a pseudoinverse via any of the three methods quite small.  Given an $r\times r$ bidiagonal matrix $B$ (notice - this derivation is for square $B$), we can uniformly tear $B$ into a block-diagonal matrix $\hat B$ with $p\times p$ blocks.  Let matrices $U,V$ be highly sparse matrices for which 
$$B = \hat B + UV^T$$
where the coefficients of $U,V$ are the severed superdiagonal coefficients $\beta_k$ and their positional basis vectors.  The Woodbury identity in this situation is 
\begin{equation}
    B^{-1} = \hat B^{-1} - \hat B^{-1}U\left( I + V^T \hat B^{-1} U\right)^{-1}V^T B^{-1}  \label{WoodIdent}
\end{equation} 
However, for a tear at row $k$ and column $k+1$, the matrix $U$ targets column $k$ while the matrix $V^T$ targets row $k+1$. But $\hat B^{-1}$ is block-diagonal and equal to 0 at row $k+1$ and column $k$.  Thus, $V^T\hat B^{-1} U = 0$ and the Woodbury identity (\ref{WoodIdent}) simplifies to 
\begin{equation} B^{-1} = \hat B^{-1} - \hat B^{-1} UV^T \hat{B}^{-1} \label{ridis} \end{equation}
Storage requirements are reduced from $O\left(r^2\right)$ to $O(r)$.  However, (\ref{ridis}) is not suitable for all applications.  If conditioning is poor or finite arithmetic problems arise, then the traditional Thomas algorithm for symmetric tridiagonal matrices might be more appropriate.

Implementations of these methods can be found at\newline \url{https://github.com/appmathdoc/DirectPseudoInverseMethods}
\newline The python implementations at that site are designed to 
\begin{itemize}
    \item avoid dense inverse matrices and thus dense MP Pseudoinverses (by mapping inverses to $O(n^2)$ back substitution implementations)
    \item be as fast or faster than traditional SVD based methods
\end{itemize}
All three methods rely on Bidiagonalization (either Golub-Kahan or Lanczos) to reduce the computation to a bidiagonal matrix.

Method 0 is simply the application of \textbf{scipy.linalg.pinv} to the bidiagonal and is included for comparison.  Method 1 is CK decomposition implemented only for bidiagonal \textbf{sympy} matrices. Because MacDuffee's formula is already excellent for computing pseudoinverses in computer algebra systems, we include the CK decomposition here not so much as a tool for calculating pseudo-inverses but as a tool for producing symbolic CK decompositions. We also include the sympy version because it seems to be especially amenable to large language model translation into other programming languages and libraries.  

Method 2 is the dual normal method.  Method 3 has two versions.  The first does not require gebrd/orgbr and is the one tested in \textbf{Method3Tests.py}  The second Method 3 is Inplace with Cython wrapping to gebrd/orbgr.  Code is provided to construct the Cython wrappings to gebrd/orbgr in LAPACK.  A large language model can combine Method 1 and Method 3 for a LAPACK version of CK decomposition for Cython bindings to gebrd/orbgr.  

The speed tests are included simply to demonstrate that these 3 methods are \emph{no slower} in general than the usual SVD pinv. Significant speedups would require implementations that do not rely on the Python interpreter as much as these implementations do.  That given, even with the Pythonic slowdown, Method 3 with Cython bindings is comparable in speed to \textbf{scipy.linalg.pinv}.  Table \ref{Table2} verifies that Method 3 is comparable in speed.   

\begin{table}[htb]
    \centering
    \begin{tabular}{c|ccc}
        m \textbackslash \; n &  10 & 100 & 1000 \\ 
        \hline
         10  & $ 0.90 \times \pm 0.17$ & $ 0.96 \times \pm 0.15$ & $ 1.23 \times \pm 0.16$ \\
        100  & $ 1.02 \times \pm 0.10$ & $ 0.64 \times \pm 0.09$ & $ 1.07 \times \pm 0.16$ \\
       1000  & $ 0.90 \times \pm 0.07$ & $ 0.67 \times \pm 0.11$ & $ 0.62 \times \pm 0.05$ 
    \end{tabular}
    \caption{Values > 1 correspond to Method 3 being that many times faster than \textbf{scipy.linalg.pinv}.}
    \label{Table2}
\end{table}

In addition,  the Lanczos bidiagonalization method can be repurposed for direct Moore-Penrose pseudoinverse calculations with error bounds as in Theorem \ref{Theorem2}.  The bound (\ref{ProofBound}) is sufficient to prove Theorem \ref{Theorem2}, but notice that
$$
Av_{j}=\alpha_{j}u_{j}+\beta_{j-1}u_{j-1}
$$
where $u_{j}$ are columns of $U_{k}$ and $v_{j}$ are columns of $V_{k}.$ 
In matrix form, this is 
$$ AV_{k}=U_{k}B_{k} $$
However, $A_{k}=U_{k}B_{k}V_{k}^{T}$ implies that 
\begin{eqnarray*}
R_{k}V_{k} &=&AV_{k}-A_{k}V_{k} \\
&=&U_{k}B_{k}-U_{k}B_{k}V_{k}^{T}V_{k}=0
\end{eqnarray*}
so that also $R_{k}A_{k}\dg =R_{k}V_{k}B_{k}U_{k}^{T}=0.$  Thus, 
(\ref{ProofBound}) reduces to 
$$
\left\Vert AX_{k}A-A\right\Vert_{F}\leq \left\Vert A_{k}A_{k}^{\dag
}R_{k}\right\Vert_{F}+\left\Vert R_{k}\right\Vert_{F}\leq \left\Vert
R_{k}\right\Vert \left( \sqrt{n}+1\right)
$$
since $A_{k}A_{k}\dg $ is an orthogonal projection.  In addition, 
(\ref{sumsquares}) can be written 
$$
\left\Vert A\right\Vert_{F}^{2}=\left\Vert U_{k}B_{k}V_{k}^{T}\right\Vert_{F}^{2}+\left\Vert R_{k}\right\Vert_{F}^{2}
$$
The Frobenius norm squared of the first term on the right is 
\begin{eqnarray*}
\left\Vert U_{k}B_{k}V_{k}^{T}\right\Vert_{F}^{2} &=&\tr\left( \left(
U_{k}B_{k}V_{k}^{T}\right) \left( U_{k}B_{k}V_{k}^{T}\right)^{T}\right) \\
&=&\tr\left( U_{k}B_{k}V_{k}^{T}V_{k}B_{k}^{T}U_{k}^{T}\right) \\
&=&\tr\left( U_{k}B_{k}B_{k}^{T}U_{k}^{T}\right) =\tr\left(
U_{k}^{T}U_{k}B_{k}B_{k}^{T}\right)
\end{eqnarray*}
As a result, (\ref{sumsquares}) reduces to 
$$
\left\Vert A\right\Vert_{F}^{2}=\left\Vert B_{k}\right\Vert_{F}^{2}+\left\Vert R_{k}\right\Vert_{F}^{2}
$$
Moreover, if $\varepsilon_{k}=\left\Vert R_{k}\right\Vert_{F}^{2}$ is the Frobenius `` energy difference" between $
\left\Vert A\right\Vert_{F}^{2}$ and $\left\Vert B_{k}\right\Vert_{F}^{2},$
then 
$$
\varepsilon_{k}=\left\Vert A\right\Vert_{F}^{2}-\sum_{i=1}^{k}\left(
\alpha_{i}^{2}+\beta_{i}^{2}\right)
$$
That is, the algorithm stops once $\varepsilon_{k}$ is sufficiently small.

Calculating the Moore-Penrose pseudo-inverse is often
application-specific.  We have augmented existing methods with direct methods as fast or faster
than using SVD-based calculations and even faster in some instances than
highly efficient least squares solvers.  Most importantly, the three methods provide
new approaches to pseudo-inverses of large, sparse matrices.

\section{Relationship to the Group Inverse}
Recall that the group inverse of a square matrix $\mathcal{A}\in \mathbb{R}^{d\times d}$ requires that $\mathcal{A}$ have an index of 1.  In \cite{benisrael2003generalized}, it is shown that $\mathcal{A
}$ has a limit of 1 only if 
$$
\lim_{\varepsilon \rightarrow 0}\left( \mathcal{A}+\varepsilon I_{n}\right)^{-1}\mathcal{A}   \quad \text{exists}
$$
in which case the group inverse $A^{\#}$ is given by 
\begin{equation}
\mathcal{A}^{\#}=\lim_{\varepsilon \rightarrow 0}\left( \mathcal{A}^{2}+\varepsilon^2 I_{n}\right)^{-1}\mathcal{A}  \label{GroupLimDef}
\end{equation}
In particular, if $\mathcal{A}$ is diagonalizable, then 
$$
\lim_{\varepsilon \rightarrow 0}\left( \mathcal{A}^{2}+\varepsilon
I_{n}\right)^{-1}\mathcal{A}=P\Lambda^{\dag }P^{-1}
$$
where $\Lambda $ is the diagonal matrix of eigenvalues of $\mathcal{A}.$ \
As a result, $\mathcal{A}^{\#}$ is the unique $n\times n$ matrix for which 
$$
\mathcal{AA}^{\#}\mathcal{A}=\mathcal{A},   \mathcal{A}^{\#}\mathcal{AA}^{\#}=\mathcal{A},  \mathcal{AA}^{\#}=\mathcal{A}^{\#}\mathcal{A}
$$

A consequence is that a symmetric matrix has a group inverse.  For example,
if $A\in \mathbb{R}^{m\times n},$ then let $d=m+n$ and let 
$$
\mathcal{A}=\left[ 
\begin{array}{cc}
0 & A \\ 
A^{T} & 0
\end{array}
\right] 
$$
The $\left( m+n\right) \times \left( m+n\right) $ matrix $\mathcal{A}$ is a 
\emph{Jordan-Wielandt Symmetric Augmented matrix }over the vector space $
\mathbb{R}^{m+n}$ \cite{golub2013matrix}.  For any $\varepsilon
\neq 0$ we have 
$$
\mathcal{A}^{2}+\varepsilon I_{n}=\left[ 
\begin{array}{cc}
AA^{T}+\varepsilon I_{n} & 0 \\ 
0 & A^{T}A+\varepsilon I_{m}
\end{array}
\right] 
$$
from which (\ref{GroupLimDef}) implies that 
$$
\mathcal{A}^{\#}=\left[ 
\begin{array}{cc}
0 & \left( A^{T}\right)^{\dag } \\ 
A^{\dag } & 0
\end{array}
\right] 
$$
The singular values of $\mathcal{A}$ are those of $A$ but each with twice
the multiplicity.  Thus, the condition number of $\mathcal{A}$ is the same
as that of $A.$

Moreover, if $A=UBV^{T}$ where $U\in \mathbb{R}^{m\times m},$ $V\in \mathbb{R
}^{n\times n}$ are orthogonal, then $\mathcal{A=WBW}^{T}$ where 
$$
\mathcal{W}=\left[ 
\begin{array}{cc}
0 & V \\ 
U & 0
\end{array}
\right] ,   \mathcal{B}=\left[ 
\begin{array}{cc}
0 & B \\ 
B^{T} & 0
\end{array}
\right] 
$$
and $\mathcal{W}$ is also orthogonal. Although $\mathcal{B}$ is neither
bidiagonal or tridiagonal, it is sparse and 
$$
 \mathcal{B}^{2}=\left[ 
\begin{array}{cc}
BB^{T} & 0 \\ 
0 & B^{T}B
\end{array}
\right] 
$$
As we showed above, we can assume $n=m+1$ when calculating pseudoinverses of
bidiagonal matrices. 

In addition, if we define the $\left( m+n\right) \times \left( m+n\right) $
orthogonal matrix  
$$
\mathcal{E}=\left[ 
\begin{array}{cc}
0 & I_{n} \\ 
I_{m} & 0
\end{array}
\right] 
$$
(recall that $n=m+1$), then $\mathcal{EB}$ is bidiagonal and of the form  
$$
\mathcal{EB=}\left[ 
\begin{array}{cc}
B^{T} & 0 \\ 
0 & B
\end{array}
\right] 
$$
with diagonal coefficients $\beta_{1},\ldots ,\beta_{m},\alpha_{1},\ldots
,\mathcal{\alpha }_{m}$ respectively. The transpose of $\mathcal{E}$ is 
$$
\mathcal{E}^{T}=\left[ 
\begin{array}{cc}
0 & I_{m} \\ 
I_{n} & 0
\end{array}
\right] 
$$
and it is straightforward to show that $\left( \mathcal{EB}\right)^{\dag }=
\mathcal{B}^{\dag }\mathcal{E}^{T}.$  

That is, Moore-Penrose pseudo-inverses are a special case of the group
inverse of a square matrix.  Thus, techniques developed for one can often
be applied to the other.  To illustrate, we notice that for
any $\varepsilon >0$ that
$$
\left( \mathcal{A}^{2}+\varepsilon^{2}I_{n}\right)^{-1}\mathcal{A=}\frac{1
}{2}\left[ \left( \mathcal{A}+i\varepsilon I_{n}\right)^{-1}+\left( 
\mathcal{A}-i\varepsilon I_{n}\right)^{-1}\right] 
$$
Consequently, if $\mathcal{A}$ is real symmetric, (\ref{GroupLimDef})
implies that its group inverse is 
\begin{equation}
\mathcal{A}^{\#}=\lim_{\varepsilon \rightarrow 0}\Re\left(   \left( 
\mathcal{A}+i\varepsilon I_{n}\right)^{-1} \right)     \label{RealReg2}
\end{equation}
Moreover, it follows that there is a family of matrices $
K_{\varepsilon },\varepsilon >0$ for which  
$$
\Re\left(   \left( \mathcal{A}+i\varepsilon I_{n}\right)^{-1}\
\right) =\left[ 
\begin{array}{cc}
0 & K_{\varepsilon } \\ 
K_{\varepsilon }^{T} & 0
\end{array}
\right] 
$$
and it also follows that  
$$
\Re\left(   \left( \mathcal{A}+i\varepsilon I_{n}\right)^{-1}\
\right) =\mathcal{A}^{\#}+\varepsilon^{2}\mathcal{P+\ldots }
$$
That is, $K_{\varepsilon }$ is actually a function of $\varepsilon^{2}$,
thus implying that the regularization converges quickly to $\mathcal{A}^{\#}.
$ This can also be generalized and is a ripe area for further exploration
and faster algorithms for calculating and approximating pseudo-inverses.

For example, for $\alpha_{1},\alpha_{2},\beta_{2},\beta_{3}$ nonzero, we
note that
$$
B=\left[ 
\begin{array}{cccc}
\alpha_{1} & 0 & 0 & 0 \\ 
0 & \alpha_{2} & \beta_{2} & 0 \\ 
0 & 0 & 0 & \beta_{3}
\end{array}
\right] \quad \implies \quad B^{\dag }=\left[ 
\begin{array}{ccc}
\frac{1}{\alpha_{1}} & 0 & 0 \\ 
0 & \frac{\alpha_{2}}{\alpha_{2}^{2}+\beta_{2}^{2}} & 0 \\ 
0 & \frac{\beta_{2}}{\alpha_{2}^{2}+\beta_{2}^{2}} & 0 \\ 
0 & 0 & \frac{1}{\beta_{3}}
\end{array}
\right] 
$$
Straightforward calculation shows that for $\varepsilon >0,$ we have
$$
\Re\left( \left( \mathcal{B}+i\varepsilon I_{n}\right)^{-1}\right) =
\left[ 
\begin{array}{cc}
0 & K_{\varepsilon }^{T} \\ 
K_{\varepsilon } & 0
\end{array}
\right] ,  \quad     K=\left[ 
\begin{array}{ccc}
\frac{\alpha_{1}}{\varepsilon^{2}+\alpha_{1}^{2}} & 0 & 0 \\ 
0 & \frac{\alpha_{2}}{\varepsilon^{2}+\alpha_{2}^{2}+\beta_{2}^{2}} & 0
\\ 
0 & \frac{\beta_{2}}{\varepsilon^{2}+\alpha_{2}^{2}+\beta_{2}^{2}} & 0
\\ 
0 & 0 & \frac{\beta_{3}}{\varepsilon^{2}+\beta_{3}^{2}}
\end{array}
\right] 
$$
That is, as long as $\varepsilon^{2}$ is sufficiently close to 0, we need
not require $\varepsilon >0$ to be close to 0.  

\section{Conclusions and Future Work}

As data science and machine learning consume larger and larger datasets in
more and more applications, the need for reliable sparse matrix friendly
algorithms is likewise increasing.  This paper has explored three methods
for calculating the Moore-Penrose pseudoinverse of a matrix $A\in \mathbb{R}^{m\times n}$ by reducing 
to a Moore-Penrose pseudoinverse of a bidiagonal matrix $B$.  We also developed a complex regularization approach that
produces good approximations to $A^{\dag }$ while avoiding conditioning
issues.  All are appropriate for sparse matrices and are at least as fast
as standard SVD-based methods for calculating pseudoinverses.  

Moreover, the three bidiagonal pseudoinverse methods -- CK decomposition, Dual Normals, and in-place
can each be combined with any bidiagonalization approach, thus implying at
least 7 different approaches are possible.  There are also a number of
variations on these variations.  Since the method used in any given
application is often specific to that application, we have presented these
methods in a way that allows a variety of related methods to be pursued. \
All allow applications to work directly with the pseudoinverse concept and
all these methods (except the regularization) calculate the pseudoinverse
directly. 

For example, Woodbury tearing and (\ref{WoodIdent}) can be applied more broadly. Given 
an $m\times n$ real matrix $A$, we can form the Jordan-Wielandt matrix $\mathcal{A}$ and can then bidiagonalize 
to obtain the $(m+n)\times(m+n)$ bidiagonal form $\mathcal{B}$.  Woodbury tearing can then be applied to  
$$\Re\left(\; \left( B+i\varepsilon\right)^{-1} \;\right) $$
for some suitable $\varepsilon > 0$.  Care must be taken in any of these variations, however, as different 
applications have varying degrees of sparsity, conditioning, and a need for speed. 

It is worth recalling that direct methods are not necessarily
the best approaches in general \emph{for all applications.  }The SVD based
approaches allow noise cutoffs and rank estimation after having calculated all
the singular values.  The CK decomposition allows a similar approach --
but only similar.  Conditioning issues are always a factor in direct
methods. 

Nonetheless, there are already applications
that require either speed or sparsity (or both) for repetitive application of a pseudoinverse. 
There likely are more to come.  In these instances, we
have shown that the methods in this paper are no worse than the traditional
SVD-based approach, and have the potential to be better in many
applications. 

\section*{Acknowledgments}
Generative AI was not used in the preparation of this manuscript.  However, the Large Language Model \emph{Gemini Pro} 
was used to generate python scripts (tests and straightforward scripts); to find bugs in author generated codes (especially those 
related to Cython LAPACK bindings); to identify similar approaches in the literature that the author 
may have overlooked; and to collaboratively produce verifications and suggestions at the request of the author.

\bibliographystyle{unsrt}
\bibliography{references} 

@book{benisrael2003generalized,
    author = {Ben-Israel, Adi and Greville, Thomas N. E.},
    edition = {2nd},
    publisher = {Springer-Verlag New York},
    title = {Generalized Inverses: Theory and Applications},
    year = {2003}
}

@article{bunch1974partitioning,
  title={Partitioning, tearing and modification of sparse linear systems},
  author={Bunch, James R and Rose, Donald J},
  journal={Journal of Mathematical Analysis and Applications},
  volume={48},
  number={3},
  pages={574--593},
  year={1974},
  publisher={Elsevier},
  doi={10.1016/0022-247X(74)90117-X}
}

@article{fuhry2011new,
  title={A new Tikhonov regularization method},
  author={Fuhry, Martin and Reichel, Lothar},
  journal={Numerical Algorithms},
  volume={59},
  pages={433--445},
  year={2011},
  publisher={Springer},
  doi={10.1007/s11075-011-9498-x}
}

@article{golub1965calculating,
    author = {Golub, Gene H and Kahan, William},
    doi = {10.1137/0702016},
    journal = {Journal of the Society for Industrial and Applied Mathematics, Series B: Numerical Analysis},
    number = {2},
    pages = {205--224},
    publisher = {SIAM},
    title = {Calculating the singular values and pseudo-inverse of a matrix},
    volume = {2},
    year = {1965}
}

@article{golub1970singular,
    author = {Golub, Gene H and Reinsch, Christian},
    doi = {10.1007/BF02163027},
    journal = {Numerische Mathematik},
    number = {5},
    pages = {403--420},
    publisher = {Springer},
    title = {Singular value decomposition and least squares solutions},
    volume = {14},
    year = {1970}
}

@book{golub2013matrix,
    address = {Baltimore, MD},
    author = {Golub, Gene H and Van Loan, Charles F},
    edition = {4th},
    isbn = {978-1421407944},
    publisher = {Johns Hopkins University Press},
    title = {Matrix Computations},
    year = {2013}
}

@article{grisetti2020least,
  title={Least squares optimization: From theory to practice},
  author={Grisetti, Giorgio and Guadagnino, Tiziano and Aloise, Irvin and Colosi, Mirco and Della Corte, Bartolomeo and Schlegel, Dominik},
  journal={Robotics},
  volume={9},
  number={3},
  pages={51},
  year={2020},
  publisher={MDPI}
}

@article{hager1989updating,
  title={Updating the inverse of a matrix},
  author={Hager, William W},
  journal={SIAM Review},
  volume={31},
  number={2},
  pages={221--239},
  year={1989},
  publisher={SIAM},
  doi={10.1137/1031049}
}

@article{higham2023power,
  title={The power of bidiagonal matrices},
  author={Higham, Nicholas J},
  journal={arXiv preprint arXiv:2311.06609},
  year={2023}
}

@article{hoefler2021sparsity,
  title={Sparsity in Deep Learning: Pruning and growth for efficient inference and training in neural networks},
  author={Hoefler, Torsten and Alistarh, Dan and Ben-Nun, Tal and Dryden, Nikoli and Peste, Alexandra},
  journal={Journal of Machine Learning Research (JMLR)},
  volume={22},
  number={241},
  pages={1--124},
  year={2021},
  note={An exhaustive survey of over 300 papers detailing how sparse matrices and network pruning are mandatory for reducing the memory footprint and energy costs of modern deep learning.}
}

@misc{mishra2021accelerating,
  title={Accelerating Sparse Deep Neural Networks},
  author={Mishra, Asit and Latorre, Jorge Albericio and Pool, Jeff and Stosic, Darko and Stosic, Dusan and Venkatesh, Ganesh and Yu, Chong and Micikevicius, Paulius},
  howpublished={arXiv preprint arXiv:2104.08378},
  year={2021},
  note={Details the hardware-level integration of sparse matrix math units (Sparse Tensor Cores) in modern NVIDIA architectures, proving that the hardware industry is pivoting to natively support sparse linear algebra.}
}

@article{paige1982lsqr,
    author = {Paige, Christopher C. and Saunders, Michael A.},
    journal = {ACM Transactions on Mathematical Software (TOMS)},
    number = {1},
    pages = {43--71},
    publisher = {ACM New York, NY, USA},
    title = {LSQR: An algorithm for sparse linear equations and sparse least squares},
    volume = {8},
    year = {1982}
}

@article{penrose1955generalized,
    author = {Penrose, Roger},
    journal = {Mathematical Proceedings of the Cambridge Philosophical Society},
    number = {3},
    pages = {406--413},
    publisher = {Cambridge University Press},
    title = {A generalized inverse for matrices},
    volume = {51},
    year = {1955}
}

@article{simon2000low,
  title={Low-rank matrix approximation using the Lanczos bidiagonalization process with applications},
  author={Simon, Horst D and Zha, Hongyuan},
  journal={SIAM Journal on Scientific Computing},
  volume={21},
  number={6},
  pages={2257--2274},
  year={2000},
  publisher={SIAM}
}

@techreport{thomas1949elliptic,
  title={Elliptic Problems in Linear Difference Equations over a Network},
  author={Thomas, Llewellyn H},
  year={1949},
  institution={Watson Scientific Computing Laboratory, Columbia University},
  address={New York, NY}
}

@inproceedings{wang2025schwarz,
  title={Schwarz-Schur Involution: Lightspeed Differentiable Sparse Linear Solvers},
  author={Wang, Yu and Abulnaga, S. Mazdak and Balbastre, Ya{\"e}l and Fischl, Bruce},
  booktitle={Proceedings of the 42nd International Conference on Machine Learning (ICML)},
  year={2025},
  note={Breaks the assumption that direct sparse solvers are too slow for neural architectures, accelerating them by orders of magnitude so they can be integrated into end-to-end AI pipelines for PDE solving.}
}

@inproceedings{zaheer2020bigbird,
  title={Big Bird: Attention with Scant Evidence},
  author={Zaheer, Manzil and Guruganesh, Guru and Dubey, Kumar Avinava and Ainslie, Joshua and Alberti, Chris and Ontanon, Santiago and Pham, Philip and Ravula, Anirudh and Wang, Qifan and Yang, Li and others},
  booktitle={Advances in Neural Information Processing Systems (NeurIPS)},
  volume={33},
  pages={17283--17297},
  year={2020},
  note={A foundational paper proving that sparse matrix attention mechanisms can approximate full dense attention, shifting the bottleneck from quadratic to linear scaling for sequence modeling.}
}
\end{document}